\documentclass[12pt]{amsart}
\usepackage{latexsym,amsmath,amsfonts,amscd,amssymb}
\usepackage{stmaryrd}
\usepackage{lscape}
\usepackage{euscript,color}
\usepackage{array}   
\usepackage{mathrsfs}
\usepackage{graphicx}
\usepackage{xypic}
\usepackage{calligra}
\usepackage[T1]{fontenc}
\theoremstyle{plain}  
\newtheorem{theorem}{Theorem}[section]
\newtheorem*{theorem*}{Theorem}

\newtheorem{proposition}[theorem]{Proposition}

\theoremstyle{remark}

\newtheorem{remark}[theorem]{Remark}
\newtheorem{remarks}[theorem]{Remarks}

\newtheorem*{claim*}{Claim}

\numberwithin{equation}{section}
\renewcommand{\leq}{\leqslant}

\renewcommand{\geq}{\geqslant}

\renewcommand{\setminus}{\smallsetminus}

\newcommand{\R}{\mathbb{R}}

\newcommand{\Z}{\mathbb{Z}}
\newcommand{\C}{\mathbb{C}}

\newcommand{\cE}{\mathcal{E}}
\newcommand{\cF}{\mathcal{F}}

\newcommand{\calH}{\mathcal{H}}
\newcommand{\cO}{\mathcal{O}}
\newcommand{\calR}{\mathcal{R}}

\newcommand{\eL}{{\mathscr{L}}}

\newcommand{\dbar}{\bar{\partial}}

\newcommand{\lra}{\longrightarrow}

\newcommand{\PSL}{\mathrm{PSL}}

\newcommand{\SU}{\mathrm{SU}}
\newcommand{\U}{\mathrm{U}}

\newcommand{\SL}{\mathrm{SL}}

\newcommand{\SO}{\mathrm{SO}}

\DeclareMathOperator{\ad}{ad}

\DeclareMathOperator{\rank}{rank}

\DeclareMathOperator{\Hom}{Hom}

\DeclareMathOperator{\Id}{Id}

\DeclareMathOperator{\Spec}{Spec}

\DeclareMathOperator{\Aut}{Aut}
\DeclareMathOperator{\Int}{Int}
\DeclareMathOperator{\Out}{Out}
\DeclareMathOperator{\Nm}{Nm}

\newcommand{\IP}{\mbox{$I\!\!\! P$}}

\newcommand{\Pic}{\operatorname{Pic}}
\renewcommand{\phi}{\varphi}

\newcommand{\liesu}{\mathfrak{su}}
\newcommand{\liesl}{\mathfrak{sl}}

\begin{document}
\dedicatory{To Nigel Hitchin on the occassion of his 70th birthday}
\title[Involutions of rank 2  Higgs bundle moduli spaces]
{Involutions of rank 2  Higgs bundle moduli spaces}

 \author[Oscar Garc{\'\i}a-Prada]{Oscar Garc{\'\i}a-Prada}
\address{Instituto de Ciencias Matem\'aticas \\
  CSIC \\ Nicol\'as Cabrera, 13--15 \\ 28049 Madrid \\ Spain}
\email{oscar.garcia-prada@icmat.es}

\author[S. Ramanan]{S. Ramanan}
\address{Chennai Mathematical Institute\\
H1, SIPCOT IT Park, Siruseri\\
Kelambakkam 603103\\
India}
\email{sramanan@cmi.ac.in}

\thanks{
  Partially supported by the Europena Commission Marie Curie IRSES
MODULI Programme PIRSES-GA-2013-61-25-34.
}


\begin{abstract}
We consider the moduli space $\calH(2,\delta)$ of 
rank 2 Higgs bundles with fixed determinant $\delta$ over a 
smooth projective curve  $X$ of genus 2 over $\C$, and 
study involutions defined
by tensoring the vector bundle with an element $\alpha$ 
of order 2 in the Jacobian of the curve,
combined with multiplication of the Higgs field by $\pm 1$. 
We describe the fixed points of these involutions  
in terms of the Prym variety  of the covering of $X$ defined
by  $\alpha$, and give an interpretation in terms of the 
moduli space of representations of the fundamental group.
\end{abstract}

\maketitle

\section{Introduction}

Let $X$ be a smooth projective curve of genus $g\geq 2$ over $\C $. A {\it Higgs bundle} 
$(E, \varphi )$ on $X$ consists of a vector bundle $E$ and a twisted endomorphism 
$\varphi :E \to E\otimes K$, where $K$ is the canonical bundle of $X$. 
The {\it slope} of $E$ is 
the rational number defined as 
$$\mu (E) = {\deg E}/{\rank E}. $$
A Higgs bundle is said to be {\it stable} (resp. {\it semistable}) if
$$\mu (F) <  ({\rm resp. } \leq )~ \mu (E)$$
for every proper subbundle $F$ of $E$ invariant under $\varphi $ in the sense that 
$\varphi (F) \subset F\otimes K$. Also, a Higgs bundle $(E, \varphi )$ is {\it polystable} if 
$(E, \varphi ) = \oplus_i (E_i, \varphi _i)$ where all the $(E_i, \varphi _i)$ are stable and 
all $E_i$ have the same slope as that of $E$. 

Let $\delta $ be a line bundle on $X$. We are interested in the moduli space $\calH(n,\delta )$ of 
isomorphism classes of polystable Higgs bundles $(E, \varphi )$ of rank $n$ with determinant $\delta $ and 
traceless $\varphi $. This moduli space was constructed analytically by Hitchin \cite{hitchin} and later 
algebraically via geometric invariant theory by Nitsure \cite{nitsure}. This space is a normal 
quasi-projective variety of dimension $2(n^2-1)(g-1)$. If the degree of
${\delta }$ and $n$ are coprime, 
$\calH(n,\delta )$ is smooth. 

Let $M(n,\delta )$ be the moduli space of polystable vector bundles of rank $n$ and determinant 
$\delta $. The set of points  corresponding to stable bundles form a smooth open set and the 
cotangent bundle of it is a smooth, open, dense subvariety of $\calH(n,\delta)$. 

In this paper, we focus on vector bundles and Higgs bundles of rank $2$, leaving the study of 
those of of higher rank (and indeed of $G$-principal bundles with $G$ reductive) for  \cite{garcia-prada-ramanan}. There are two kinds of involutions that we consider. Firstly the subgroup 
$J_2$ of elements of the Jacobian $J$ consisting of elements of order $2$ acts on $\calH(2,\delta )$ 
by tensor product. We also consider the involutions where in addition, the sign of the Higgs 
field is changed. More explicitly, for $\alpha\in J_2$ we consider 
the involutions

\begin{equation}\label{involutions}
   \begin{aligned}
\iota(\alpha)^\pm: \calH(2,\delta) & \to \calH(2,\delta) \\
(E,\varphi) & \mapsto (E\otimes\alpha,\pm \varphi).
  \end{aligned}
\end{equation}

We determine the fixed point varieties in all these cases, and their
corresponding subvarieties of the moduli space of representations of
the fundamental group of $X$ (and its universal central extension) 
under the correspondence between this moduli
space and the moduli space of Higgs bundles,
established by Hitchin \cite{hitchin} and Donaldson \cite{donaldson}.
The case of the involution $(E,\varphi)\mapsto (E,-\varphi)$ is already 
covered in the beautiful paper of Hitchin \cite{hitchin}.
 
\section{Line bundles}

To start with, we consider involutions in the case of line bundles. The moduli space of line bundles of degree $d$ is 
the {\it Jacobian variety} $J^d$. There is a universal line bundle (called a Poincar{\'e} 
bundle) on $J^d\times X$ which is unique up to tensoring by a line bundle pulled back from $J^d$. 
We will denote $J^0$ simply by $J$.

The involution $\iota :L \to L^{-1}$ of $J$ has obviously the finite set $J_2$ of elements of order 
$2$, as its fixed point variety. 

The Higgs moduli space of line bundles consists of pairs $(L, \varphi )$ where $L$ is a line bundle
of fixed degree and $\varphi $ is a section of $K$. The moduli space of rank 1
Higgs bundles of degree $d$ is thus isomorphic to $J^d\times H^0(X,K)$.
There are a few involutions to consider even in 
this case. Firstly on the Higgs moduli space of line bundles of degree $d$, one may consider the 
involution $(L, \varphi ) \to (L, -\varphi )$. The fixed point variety is just $J^d$ imbedded in the 
Higgs moduli space by the map $L \mapsto (L, 0)$ since any automorphism of $L$ induces identity on the set 
of Higgs fields on $L$. 

When $d = 0$, one may also consider the involution $(L, \varphi ) \mapsto (L^{-1}, \varphi )$. This 
has as fixed points the set $\{ (L, \varphi ): L \in J_2 \;\;\mbox{and}\;\;
\varphi\in H^0(X,K)\} $. Also, we may consider the composite 
of the two actions, namely $(L, \varphi ) \mapsto (L^{-1}, -\varphi )$. Again it is obvious that the 
fixed points are just points of $J_2$ with Higgs fields $0$.   

Finally, translations by elements of $J_2\setminus \{ 0 \}$ are involutions without fixed points. 

\section{Fixed Points of $\iota(\alpha)^{-}$}\label{triples}

We wish now to look at involutions of $M = M(2, \delta )$ and $\calH = \calH(2, \delta )$. We will often 
assume that $\delta $ is either ${\cO}$ or a line bundle of degree $1$. There is no loss of 
generality, since the varieties $M$ and $\calH$ for any $\delta $ are isomorphic 
(on tensoring with a suitable line bundle) to ones with $\delta $ as above. In general, we denote 
by $d$ the degree of $\delta $.

If $d$ is odd, the spaces $M$ and $\calH$ are smooth and the points correspond
to stable bundles  and stable Higgs bundles, respectively. If $d$ is even (and $\delta $ trivial), 
there is a natural morphism  $J \to M$ which takes $L$ to $L\oplus L^{-1}$ and imbeds the quotient of $J$ by the involution 
$\iota$ on $J$, namely the Kummer variety, in $M$. This is the non-stable locus (which is also the 
singular locus if $g > 2$) of $M$ and has $J_2$ as its own singular locus. 

\begin{remark}
If $(E, \varphi )\in \calH$, but $E$ is not semi-stable, then there is a line sub-bundle $L$ 
of $E$ which is of degree $>d/2$. Moreover, it is the unique sub-bundle with degree $\geq d/2$. 
Clearly, since $(E, \varphi )$ is semi-stable, $\varphi $ does not leave $L$ invariant. Hence 
$(E, \varphi )$ is actually a stable Higgs bundle. In particular, it is a
smooth point of $\calH$.
\end{remark} 


Before we take up the study of the involutions (\ref{involutions}) in general,
we note that even  when $\alpha$ is trivial, the 
involution $\iota^-:=\iota(\cO)^{-}$ is non-trivial and is of interest. In this case, the fixed point 
varieties were determined by Hitchin \cite{hitchin} and we recall the results
with some additions and  clarifications.

\begin{proposition}
Polystable Higgs bundles $(E, \varphi )$ fixed by the involution 
$\iota^{-}:(E, \varphi ) \mapsto (E, -\varphi )$ fall under the following types:

\begin{itemize}

\item[(i)] $E\in M = M(2, \delta )$ and $\varphi = 0$.

\item[(ii)] For every integer $a$ satisfying $ 0 <  2a - d \leq 2g - 2$, consider the set $T_a$
of triples 
$(L, \beta,\gamma )$ consisting of a line bundle $L$ of degree $a$ and
  homomorphisms 
$\beta : L^{-1}\otimes \delta \to L\otimes K$, with $\gamma \neq 0$  
and 
$\gamma : L\to L^{-1}\otimes
  \delta \otimes K$.

\item[(iii)] Same as in ii), but with $2a = d$ if $d$ is even.

To every triple as in ii) or iii),  associate the Higgs bundle $(E, \varphi )$ where 

\begin{equation}\label{higgs-bundle}
E = L\oplus (L^{-1}\otimes \delta ) \;\;\;\; \;\; \mbox{and}\;\;\;\;\;\; 
\varphi =
\begin{pmatrix}
  0 & \beta \\
  \gamma & 0
\end{pmatrix}.
\end{equation}

\end{itemize}

Any type {\em (ii)} Higgs bundle $(E, \varphi)$ is stable whereas $E$ is not even semi-stable. 
 
In type {\em (iii)} if $L^2$ is not isomorphic to $\delta $, and 
$\beta $ and $\gamma $ are both non-zero, then
$(E, \varphi )$ is stable. If $L^2\cong \delta$ and 
$\beta $ and $\gamma $ (both of which are then sections of $K$) are linearly independent,
then $(E, \varphi )$ is stable.
\end{proposition}

\begin{proof}

Firstly, if $E\in M$ and $\varphi = 0$, it is obvious that it is fixed under the above 
involution. On the other hand, it is clear that if $(E, \varphi )$ is of type (ii) or (iii), then 
the automorphism of $E$ 

\begin{equation}\label{i-matrix}
\begin{pmatrix}
  i & 0 \\
  0  & -i
\end{pmatrix},
\end{equation}

\noindent
takes $\varphi $ to $-\varphi $. In type (ii),  since $2a - d > 0$, it follows that $L$ is the only line
sub-bundle of $E$ of degree $\geq d/2$. Since $(E, \varphi )$ is semi-stable, $L$ is not invariant 
under $\varphi $,  (which is the case if and only if $\gamma $ is non-zero). Therefore,  $(E, \varphi )$ is 
stable. 

Type (iii) is relevant only when $d$ is even and so we will assume that $\delta $ is trivial. 
If $L^2$ is not trivial, then every line subbundle 
of $E$ of degree $0$ is either $L$ or $L^{-1}$. Since 
we have assumed that $(E, \varphi )$ is poly-stable, either $\varphi $ leaves both 
$L$ and $L^{-1}$ invariant or neither, i.e. $\beta $ and $\gamma $ are both zero or 
both non-zero. The former case is covered under type i) and in the latter case, $(E, \varphi )$ is stable. 
Finally, if $L^2$ is trivial, then every line sub-bundle of degree $0$ is isomorphic 
to $L$, and all imbeddings of $L$ in 
$E = L \oplus L$ are given by $v \mapsto (\lambda v , \mu v)$, with $(\lambda , \mu ) \neq 0$. The 
restriction of $\varphi $ to $L$ composed with the projection of $E\otimes K$ to 
$(E/L) \otimes K = (L\otimes K)$, is given by $\lambda \gamma + \mu \beta $. Hence this imbedding of 
$L$ is invariant under $\varphi $ if and only if $\lambda \gamma + \mu \beta = 0$, proving that if 
$\beta $ and $\gamma $ are linearly independent, then $(E, \varphi )$ is stable.  Otherwise, $(L,0)$
is a (Higgs) subbundle of $(E, \varphi )$ and hence it is covered again in i).

Conversely, let $(E, \varphi )$ be a {\it stable} Higgs bundle fixed by the involution. Then there exists 
an automorphism $f$ of $E$ (of determinant 1) which takes $\varphi $ to $-\varphi $. If $E$ is a stable 
vector bundle, all its automorphisms are scalar multiplications which take $\varphi $ into itself. Hence 
$\varphi = 0$ in this case. Let $E$ be nonstable. Obviously, then $\varphi $ is non-zero. Since $f^2$ 
is an automorphism of 
the stable Higgs bundle $(E, \varphi )$, we have $f^2 = \pm \Id_E$. This implies that $f_x$ is semi-simple 
for all $x \in X$. If $f^2 = \Id_E$, the eigenvalues of $f_x$ are $\pm 1$ and since $\det(f_x) = 1 $  we 
have $f = \pm \Id_E$ which would actually leave $\varphi $ invariant. So $f_x$ has $\pm i$ as eigenvalues 
at all points. We conclude that $E$ is a direct sum of line bundles corresponding to the eigenvalues 
$\pm i$. Thus $f^2 =-\Id_E$ and $E = L \oplus (L^{-1}\otimes \delta )$ with $f|L = i.\Id_E$, and 
$f|(L^{-1}\otimes \delta ) = -i.\Id_E$. We may assume that $\deg L = a \geq d/2$, replacing $L$ by 
$L^{-1}\otimes \delta $ (and $f$ by $-f$) if necessary. If $ a > d/2$, it also follows that the composite
of $\varphi |L $ and the projection $E\otimes K \to L^{-1}\otimes \delta \otimes K$ is nonzero (since
$(E, \varphi )$ is semi-stable) which implies that $a \leq -a + d +2g -2$ , i.e . $2a - d \leq 2g - 2$.  
Moreover, from the fact that $f$ takes $\varphi $ to $-\varphi $, one deduces that $\varphi $ is of the form 
claimed. 

If $(E, \varphi )$ is not stable, in which case we may assume $\delta $ is trivial, $(E, \varphi )$ is a direct 
sum of $(L, \psi )$ and $(L^{-1}, -\psi )$ with $\deg L = 0$. If $\psi $ is nonzero, then $(E, \varphi )$ is 
isomorphic to $(E, -\varphi )$ if and only if $L\cong L^{-1}$. If then $L\cong L^{-1}$ we may take 
$g =1/\sqrt 2\begin{pmatrix}1&1\\-1&1\end{pmatrix}$ and change the decomposition of $E$ to 
$g(L) \oplus g(L)$ and see that $(E, \varphi )$ falls under type (iii). 
\end{proof}

\subsection{The set of triples}

The above proposition leads us to consider the set of triples as in type (ii) and type (iii) 
above with $d \leq 2 a \leq  d + 2g - 2$. Set $m = 2a - d$. 
To such a triple, we have associated  the Higgs bundle $(E,\varphi )$ given by
$E = L \oplus (L^{-1}\otimes \delta )$ and $\varphi $ by the matrix in
(\ref{higgs-bundle}).

Notice however that this triple and the triple $(L,  \lambda ^{-1}\beta,\lambda \gamma)$ give rise to 
isomorphic Higgs bundles. So we consider the set of triples $(L, \beta , \gamma )$ as above, make $\C ^*$ 
act on it, in which $\lambda \in \C ^*$ takes $(L, \beta , \gamma )$ to 
$(L, \lambda ^{-1}\beta, \lambda \gamma)$ and pass to the quotient. We have thus given an injective 
map of this quotient into the $\iota ^{-}$-fixed subvariety of Higgs bundles.  

We will equip this quotient with the structure of a variety. 

\subsection{Construction of the space of triples.}

Take any line  bundle ${\eL}$ on $T \times X$, where $T$ is any parameter variety. 
For any $t\in T$, denote by ${\eL}_t$ the line bundle $\eL|{{t}\times X}$. Assume that 
$\deg ({\eL}_t) = r$ for all $t\in T$. Then we get a (classifying) morphism $c_{\eL}:T\to J^r$ 
mapping $t$ to the isomorphism class of ${\eL}_t$. 

There is a natural morphism 
of $S = S^r(X) \to J^r$ since $S\times X$ has a universal divisor giving rise
to a family of line bundles on $X$ of degree $r$, parametrised by $S$.
The pull back of any Poincar{\'e} bundle on $J^r\times X$ to $S\times X$ is the tensor 
product of the line bundle given by the universal divisor on $S \times X$ and 
a line bundle pulled back from $S$. The composite of the projection of this
line bundle $U$ to $S$ and the morphism $S\to J$ blows down the zero section of
the line bundle to $Z$ and yields  actually an affine morphism 
and the fibre over any $L\in J^r$ can be identified with $H^0(X,L)$, coming up with
a section $Z$ of this affine morphism.  Notice that
if $r > 2g -2$, this is actually a vector bundle over $J^r$ of rank $r + 1 - g$ and $Z$ is its
zero section. 

If $\eL$ is a family of line bundles of degree $r$ on $X$, parametrised by $T$ as above 
the pull back of  the morphism $U \to J^r$ by $c_E:T \to J^r$ will be denoted $A({\eL})$. 

If $m >0$, let $V$ be the pull back by the map $J^a \to J^{2g - 2  +  m}$ given by
$L\to K\otimes L^2  \otimes \delta ^{-1}$ of the above vector bundle.
On the other hand the map $L \to K \otimes L^{-2} \otimes \delta $ of $J^a \to J^{2g - 2 - m}$
pulls back the symmetric product $S^{2g - 2 -m}$ and gives a $2^{2g}$-sheeted \'etale covering.
The inverse image of $V$ tensored with a line bundle on $S^{2g -2  -m}$ thus gives the 
required structure on the quotient of $T_a$ by $\C ^*$. 

\begin{proposition} For each $m$ with $0 < m < 2g -2$, consider the pull-back of the map $S^{2g - 2 - m}
\to J^{2g - 2 - m}$ by the map $L \to K \otimes L^{-2}\otimes \delta $. A vector bundle over this of rank
$g-1 + m$ is isomorphic to a subvariety of Higgs bundles which are all fixed by $\iota^-$. 
\end{proposition}

We have seen that $M$ imbedded in $H$ by $E$ to $(E, 0)$ is a fixed point variety. It is of course closed
and in fact, compact as well. 

The set of type (ii) fixed points is the disjoint union of $T_a$ with $d/2 < a < g - 1 + d$ and (disjoint 
from $M$ as well). Each of these 
gives an injective morphism of a vector bundle on a $2^{2g}$-sheeted \'etale covering of $S^a$ into the fixed point
subvariety.  Since the subvariety of $H$ corresponding to nonstable vector bundles is  smooth and closed, 
this morphism is an isomorphism onto the image. 

We need to describe the image of the subvariety $T_a$ when $a = d/2$. We will assume $d =0$ and 
$\delta $ is trivial. Consider the  natural map of 
$S^{2g -2}$ onto $J^{2g -2}$. Pull it back to $J$ by the two maps $L \to
K\otimes L^2$ and $L\to K\otimes L^{-2}$. Take their fibre product and the quotient by the involution which
changes the two factors.  There is a natural map of this quotient into 
${\IP}H^0(K^2)$.
Pull back the line bundle $\cO(1)$ on ${\IP}H^0 (K^2)$ to this . It is easy to check that this is 
irreducible and closed. 

There are other irreducible components of type (iii)  in the case of $g = 2$.
Take any line bundle $L$ of order 2 and  consider 
\begin{equation}
\begin{pmatrix}
  0 & \beta \\
  \gamma & 0
\end{pmatrix}
\end{equation}
as a Higgs field on $L \oplus L$. Consider the tensor product map $\beta \otimes \gamma$ 
into $H^0 (K^2)$. This is surjective and can be identified with the quotient by $\C ^*$ and 
$\Z/2$ of the fixed point set given by $(L, \beta , \gamma)$  with $L\in J_2$.

\subsection{An Alternative point of view}

Note that both in Type (ii) and Type (iii) we have a natural morphism of these components into 
$H^0(X, K^2)$ given by $(\beta , \gamma )\mapsto -\beta\gamma $. Clearly this is the restriction of the Hitchin map. 
Given a (non-zero) section of $H^0 (K^2)$ we can partition its divisor into two sets of cardinality
$2g -2 -m $ and $2g -2 + m$. They yield elements of $J^{2g -2 -m}$ and $J^{2g -2 +m}$ together with 
non-zero sections $\beta $ and $\gamma $ which are defined up to the action of $\C^*$ as we have defined 
above.  Passing to a $2^2g$-sheeted \'etale covering we get the required set.
In particular it follows that except in case i) when the Hitchin map is $0$, in all other cases, the 
Hitchin map is finite and surjective.

\section{Prym varieties and rank 2 bundles}\label{prym}

Let now $\alpha \in J_2\setminus \{0\}$. To start with, we will determine 
the fixed points of the involution defined on $M$ defined by tensoring by
$\alpha$

\begin{proposition}
Let $E$ be a vector bundle of rank $2$ on $X$, and let $\alpha $ be a non-trivial line bundle
of order $2$ such that $(E\otimes \alpha )\cong E$. Then $E$ is polystable. Moreover if $E$ is not stable 
it is of the form $L \oplus (L\otimes \alpha )$ with $L^2 \cong\alpha $.
\end{proposition}
\begin{proof}
Assume that $(E\otimes \alpha )\cong E$. If $E$ is not poly-stable, then it has a unique line subbundle $L$ of 
maximal degree. This implies that $(L\otimes \alpha )\cong L$ which is absurd. If $E$ is of the form $L \oplus M$, 
then under our assumption, it follows that $M \cong L\otimes \alpha $.
\end{proof}

We recall \cite{mumford,narasimhan-ramanan} the relation between the Prym
variety of a two-sheeted  {\'e}tale cover of $X$ and vector 
bundles of rank 2 on $X$. If $\alpha $ is a non-trivial element of $J_2(X)$, there is associated to it a canonical 
$2$-sheeted {\' e}tale cover $\pi :X_{\alpha } \to X$, namely $\Spec({\cO} \oplus \alpha )$ with the 
obvious algebra structure on this locally free sheaf. Let $\iota$ be the Galois involution. For every line bundle $L$ 
of degree $d$ on $X_\alpha$, the line bundle $L\otimes \iota^*L$ of degree $2d$ with the natural lift of $\iota$ clearly descends 
to a line bundle $\Nm(L)$ of degree $d$ on $X$. This gives the {\it norm homomorphism} $\Nm:\Pic(X_\alpha)\to \Pic(X)$. Its 
kernel consists of two components and the one which contains the trivial line bundle is the {\it Prym variety} 
$P_{\alpha }$ associated to $\alpha $. If $L$ is a line bundle on $X_\alpha$, its direct image $\pi _*(L)$ on $X$ 
is a vector bundle of rank $2$. Note that $\det(\pi _*({\cO})) = \det({\cO}\oplus \alpha) = \alpha$, and 
more generally that $\det(\pi _*(L)) = \Nm(L) \otimes \alpha $ for all $L$. The fibres of $\Nm$ consist of two 
cosets $F_\alpha$ of $P_{\alpha }$ and the Galois involution interchanges the 
two if the degree is odd and leaves each 
component invariant if the degree is even. In particular, it acts on $P_{\alpha }$, and indeed as 
$L\mapsto L^{-1}$ on it. 

\begin{proposition}
For any line bundle $L$ on $X_\alpha$, the direct image $E = \pi _*L$ 
is a polystable vector 
bundle of rank $2$ on $X$ such that $E\otimes \alpha \cong  E $. If $E$ is not stable, it is of the form 
$\xi \oplus (\xi \otimes \alpha)$.
\end{proposition}

\begin{proof}
Indeed, if $\xi $ is any line subbundle of $E$, its inclusion in $E$ gives 
rise to a nonzero homomorphism $\pi ^*\xi  \to L$,  and hence 
$2\deg \xi = \deg (\pi ^*\xi ) \leq \deg(L) = \deg(E)$, proving $E$ is
semi-stable. If $\deg \xi = \deg E/2$, 
the homomorphism $\pi ^*\xi \to L$ is an isomorphism. But then 
$\pi _*L = \pi _*(\pi ^*\xi ) = \xi \otimes \pi _*{\cO} = \xi \otimes ({\cO} \oplus \alpha )$ 
proving our assertion.
\end{proof}

We have thus a morphism of $\Nm^{-1}(\delta \otimes \alpha )$ into $M(2, \delta )$ which maps $L$ to $\pi _*L$. 
Let $E$ be stable such that $E\otimes \alpha \cong E$. we may then choose an isomorphism 
$f:E \to E\otimes {\alpha }$ such that its iterate $(f\otimes \Id_{\alpha})\circ f: E\to E$ is the identity. 
Indeed this composite is an automorphism of $E$ and hence a non-zero scalar. We can then replace the isomorphism 
by a scalar multiple so that this composite is $\Id_E$. Now the locally free sheaf ${\cE}$ can be provided a 
module structure over ${\cO}\oplus \alpha $ by using the above isomorphism. This means that it is the direct 
image of an invertible sheaf on $X_\alpha$. On the other hand, if $E$ is poly-stable but not stable, it is isomorphic 
to $L\oplus M$. If $E\otimes \alpha $ is isomorphic to $E$, it follows that $L\cong M \otimes \alpha $. Hence 
we deduce that the above morphism $\Nm^{-1}(\delta\otimes \alpha )\to M(2, \delta )$ is onto the fixed point variety under the 
action of tensoring by $\alpha$ on $M(2, \delta )$. 
If $\pi _*L \cong \pi _*L'$, then by applying 
$\pi ^*$ to it, we see that $L'$ is isomorphic either to $L$ or $\iota^*L$. In other words, the above map descends to an isomorphism of the quotient of $\Nm^{-1}(\delta\otimes\alpha )$ by the Galois involution onto the $\alpha $-fixed 
subvariety of $M(2, \delta)$. Since the fibres of $\Nm$ are interchanged by the Galois involution when $\delta$ 
is of odd degree, this fixed point variety is isomorphic to a coset of the Prym variety. When $\delta $ is of 
even degree, the $\alpha $-fixed variety has two connected components, each isomorphic to the quotient of the 
Prym variety by the involution $L\to L^{-1}$, that is to say to the Kummer variety of Prym. We collect these facts 
in the following.

\begin{theorem}\label{fixed-points-M}
Let $\alpha $ be a non-trivial element of $J_2(X)$. It acts on $M(2, \delta )$ by tensor 
product: $\iota(\alpha)(E):=E\otimes \alpha$. The fixed point variety $F_\alpha(\delta)$ is isomorphic to the 
Prym variety of the covering $\pi:X_{\alpha }\to X$ given by 
$\alpha $ if $d = \deg \delta$ is odd, and is isomorphic to the union of two irreducible components, each 
isomorphic to the Kummer variety of the Prym variety, if $d$ is even.
\end{theorem}

\begin{remarks}
(1) If $L$ is a line bundle on $X_{\alpha }$ and $E = \pi _*L$, then since $E\otimes \alpha \cong E$, 
$\alpha $ is a line sub-bundle of $\ad(E)$. Indeed, since $E$ is poly-stable, $\alpha $ is actually a direct 
summand. To see this, interpret $\ad(E)$ as $S^2(E)\otimes \det(E)^{-1}$ and 
notice that there is a 
natural surjecion of $S^2(\pi_* L)$ onto $\pi_*(L^2)$.  It follows that $\pi_*(L^2) \det(E)^{-1}$ is contained in 
$\ad(E)$. Thus we see that 
$$\ad (\pi _*L) \cong \alpha \oplus ((\pi _*L^2)\otimes  \alpha \otimes \Nm(L^{-1})).$$ 
(2) As we have seen above, in the case $\delta $ is trivial, the fixed point variety intersects the 
non-stable locus, namely the Kummer variety of the Jacobian at bundles of the form 
$\xi \oplus (\xi \otimes \alpha )$, where $\xi $ is a line bundle with $\xi ^2 \cong \alpha $. Clearly, $\xi $ 
and $\xi \otimes \alpha $ give the same bundle. Thus the intersection of the two copies of the Prym Kummer variety 
(corresponding to any non-trivial $\alpha \in J_2$) with the Jacobian--Kummer variety is an orbit of smooth points, 
under the action of $J_2$. This geometric fact can be stated in the context of principally polarised abelian 
varieties and is conjectured to be characteristic of Jacobians. Analytically expressed, this is the Schottky 
equation.
\end{remarks}

\section{Fixed Points of $\iota(\alpha )^{\pm }$ when $d$ is odd.}\label{fix-odd}

If $(E, \varphi )$ is a polystable Higgs bundle fixed under either of the involutions $\iota(\alpha )^{\pm }$, we
observe that $E$ is isomorphic to $E\otimes \alpha $. This implies that $E$ is itself polystable. 
Hence if $d$ is odd, we have only to consider the action of $\alpha$ on $M =
M(2, \delta )$, given by $E\mapsto E\otimes \alpha$,  and look at
its action on the cotangent bundle. Let $F_\alpha$ be the fixed point variety
in $M$ under the  action of $\alpha$ (see Theorem \ref{fixed-points-M}), we 
have the exact sequence 
$$0\to N(F_\alpha, M)^* \to T^*(M)|_{F_\alpha} \to T^*(F_\alpha) \to 0,$$ 
where $N(F_\alpha,M)$ is the normal bundle of $F_\alpha$ in $M$. This 
sequence splits canonically since $\alpha $ acts on the 
restriction of the tangent bundle of $M$ to $F_\alpha$ and splits it into eigen-bundles corresponding to the eigen-values 
$\pm 1$. Clearly the subbundle corresponding to the eigen-value $+1$ (resp. -1) is $T(F_\alpha)$ (resp. $N(F_\alpha, M)$).  
Since $E\otimes \alpha \cong E$ and $d$ is odd, $E$ is stable, and we have
the following.

\begin{theorem}
If $\deg \delta$ is odd the fixed point subvariety $\cF_\alpha^+$ (resp.  $\cF_\alpha^-$)
of the action of 
$\iota(\alpha )^+$ (resp. $\iota(\alpha )^{-}$) on $\calH(2,\delta)$ is the 
cotangent bundle $T^*(F_\alpha)$ of $F_\alpha$ 
(resp. the conormal bundle  $N(F_\alpha,M)^*$ of $F_\alpha$). 
\end{theorem}

\section{Fixed points of $\iota(\alpha )^\pm$ when $d$ is even.}\label{fix-even}

We may assume that the determinant is trivial in this case. 
If $(E, \varphi )$ is fixed by either of the 
involutions $\iota(\alpha )^{\pm }$, with $E$ stable, 
the above discussion is still valid so that we have 
\begin{itemize}
\item[(i)] The sub-variety of fixed points of $\iota(\alpha )^+$ is $T^*(F_\alpha^{stable})$.
\item[(ii)] The sub-variety of fixed points of $\iota(\alpha )^{-}$ is $N^*(F_\alpha^{stable}, M)$. 
\end{itemize}

Assume then that $(E,\varphi)$ is a fixed point of $\iota(\alpha )^{\pm }$, 
where $E$ is  polystable of the form $L \oplus L^{-1}$. 
We have $L^{-1}\cong L \otimes \alpha $ 
and $\varphi $ is of the form 

\begin{equation}\label{higgs-field}
\varphi= 
\begin{pmatrix}
  \omega & \beta \\
  \gamma  & -\omega
\end{pmatrix},
\end{equation}

\noindent
with $\beta,\gamma \in H^0(K\otimes \alpha )$ and $\omega \in H^0(K)$. 
Since the summands of $E$ are 
distinct, any isomorphism $f: E \otimes \alpha \to E$ has to be of the form 

$$
\begin{pmatrix}
0 & \lambda \\
-\lambda^{-1}  & 0
\end{pmatrix},
$$

\noindent
with $\lambda \in \C^*$. Also, $f$ takes $\varphi $ to $\pm \varphi $ 
if and only if

$$
\begin{pmatrix}
  0 & \lambda \\
  -\lambda^{-1}  & 0
\end{pmatrix}
\begin{pmatrix}
  \omega & \beta \\
  \gamma  & -\omega
\end{pmatrix}
\begin{pmatrix}
  0 & -\lambda \\
  \lambda^{-1}  & 0
\end{pmatrix}
=\pm
\begin{pmatrix}
  \omega & \beta \\
  \gamma  & -\omega
\end{pmatrix}.
$$
In other words,

\begin{equation}\label{condition}
\begin{pmatrix}
  -\omega & \lambda^{-2}\gamma \\
  \lambda^{2}\beta  & \omega
\end{pmatrix}
=\pm
\begin{pmatrix}
  \omega & \beta \\
  \gamma  & -\omega
\end{pmatrix}.
\end{equation}

We analyse the cases $\iota(\alpha)^+$ and $\iota(\alpha)^-$ separately.

\subsection{Fixed points of  $\iota(\alpha)^+$}

In the case of $\iota(\alpha )^{+}$, (\ref{condition}) implies 
that $\omega = 0$ and $\lambda ^2 \beta =  \gamma $.
If $\beta $ or $\gamma $ is $0$, the Higgs bundle is $S$-equivalent 
to $(L \oplus (L\otimes \alpha ), 0)$. 
Hence this fixed point variety is isomorphic to the
product of $J/\alpha $ and the 
 space of decomposable tensors in $H^0 (K) \otimes 
H^0(K)$.

\begin{remark}
Since $E$ is of the form $\pi _*(L)$, we conclude that 
the tangent space at $E$ to $M$ (assuming that $E$ is stable) is 
$$H^1(\ad (E)) = H^1(\alpha ) \oplus H^1(\pi_* (L^2)\otimes \alpha ).$$
It is clear that the first summand here is the tangent space to the Prym variety while the second is the space 
normal to Prym in $M$.
\end{remark}

\subsection{Fixed points of  $\iota(\alpha)^-$}
It is clear that 
\begin{equation}
\begin{pmatrix}
  \omega & \beta \\
  \gamma  & -\omega
\end{pmatrix}.
\end{equation}
is taken to its negative under the action of 
\begin{equation}
\begin{pmatrix}
  0 & \lambda \\
 - \lambda^{-1}  & 0
\end{pmatrix}
\end{equation}
if and only if $\beta $ and $\gamma $ are multiples of each other, in which case we
may as well assume that $\beta = \gamma $. 
In other words, $E$ belongs to the Prym variety and $\varphi $ belongs to 
$H^0(K) \oplus H^0(K \otimes \alpha )$.

\section{Higgs bundles and representations of the fundamental group}

Let  $G$ be  a reductive Lie group, and let $\pi_1(X)$ 
be the fundamental group of $X$.
A representation $\rho:\pi_1(X)\lra G$ is said to be {\em reductive}
if the composition of $\rho$ with the adjoint representation of $G$ in its 
Lie algebra is completely reducible. When  $G$ is algebraic, this is 
equivalent to the Zariski closure of the image of $\rho$ being a reductive 
group. If $G$ is compact or abelian every representation is reductive.
We thus define the \emph{moduli space of representations} of
$\pi_1(X)$ in $G$ to be the orbit space
$$
\mathcal{R}(G) = \Hom^{red}(\pi_1(X),G) / G 
$$
of reductive representations. With the quotient topology, $\calR(G)$ has the
structure of an algebraic variety.

In this section  we  briefly 
review the relation between rank 1 and  rank 2 Higgs
bundles, and  representations of the fundamental group of the surface  
and its universal central extension in $\C^\ast$, $\U(1)$, $\R^*$,
$\SL(2,\C)$, $\SU(2)$ and
$\SL(2,\R)$. For more details, see 
\cite{hitchin,donaldson,corlette,narasimhan-seshadri}.

\subsection{Rank 1 Higgs bundles and representations}
As is well-known $\calR(U(1))$ is in bijective correspondence with the 
space $J$ of isomorphism classes of line 
bundles of degree $0$. Also, if we identify $\Z /2$ with the subgroup 
$\pm 1$ in $U(1)$ we get a bijection of 
$\calR(\Z /2)$ with the set $J_2$ of line bundles of order $2$. 

By Hodge theory one  shows that $\calR(\C^*)$ is in bijection with
$T^*J\cong J\times H^0(X,K)$, the moduli space of Higgs bundles or rank 1 
and degree $0$.  The subvariety of  fixed points of the involution 
$(L,\varphi)\to (L^{-1},\varphi)$ in this moduli space is  $J_2\times H^0(X,K)$
and corresponds to
the subvariety $\calR(\R^*)\subset \calR(\C^\ast)$.

\subsection{Rank 2 Higgs bundles and representations}

The notion of stability of a Higgs bundle $(E,\varphi)$ emerges as 
a condition for the existence of a Hermitian metric on $E$ satisfying
the Hitchin equations. More precisely, Hitchin \cite{hitchin} proved the 
following.

\begin{theorem} \label{hk} An $\SL(2,\C)$-Higgs bundle $(E,\varphi)$ 
is polystable if and only if $E$ admits a
hermitian metric $h$ satisfying
$$
F_h +[\varphi,\varphi^{*_h}]= 0,
$$
where $F_h$ is the curvature of the Chern connection defined by $h$.
\end{theorem}

Combining Theorem \ref{hk} with a theorem of Donaldson \cite{donaldson} about 
the  existence of a harmonic metric on a flat $\SL(2,\C)$-bundle with 
reductive holonomy representation, one has 
the following non-abelian generalisation of the Hodge correspondence explained
above for the rank 1 case \cite{hitchin}.

\begin{theorem}\label{correspondence}
The varieties $\calH(2,\cO)$ and $\calR(\SL(2,\C))$ are homeomorphic.
\end{theorem}

The representation $\rho$ corresponding to a polystable Higgs bundle
is the holonomy representation of the flat $\SL(2,C)$-connection given
by
\begin{equation}\label{higgs-connection}
D=\dbar_E+\partial_h+\varphi+\varphi^{*_h},
\end{equation}
where $h$ is the solution to Hitchin equations and $\dbar_E+\partial_h$ is
the $\SU(2)$-connection  defined by $\dbar_E$, the Dolbeault
operator of $E$ and $h$.

\begin{remark}
Notice that the complex structures of $\calH(2,\cO)$ and $\calR(\SL(2,\C))$ are
different. The complex structure of $\calH(2,\cO)$ is induced by the complex
structure of $X$, while that of $\calR(\SL(2,\C))$ is induced by the complex
structure of $\SL(2,\C)$.
\end{remark}

Higgs bundles with fixed determinant $\delta$ of odd degree can also
be interpreted in terms of representations. For this we need to consider
the universal central extension of $\pi_1(X)$ 
(see \cite{atiyah-bott,hitchin}). Recall that
the  fundamental group, $\pi_1(X)$, of $X$ is a finitely generated group
generated by $2g$ generators, say $A_{1},B_{1}, \ldots, A_{g},B_{g}$,
subject to the single relation $\prod_{i=1}^{g}[A_{i},B_{i}] = 1$. 
It has a universal central extension

\begin{equation}\label{eq:gamma}
0\lra\Z\lra\Gamma\lra\pi_1(X)\lra 1 \
\end{equation}

\noindent generated by the same generators as $\pi_1(X)$, together with a
central element $J$ subject to the relation $\prod_{i=1}^{g}[A_{i},B_{i}] = J$.

Representations of $\Gamma$ into $\SL(2,\C)$ are of two types
depending on whether the central element $1\in \Z\subset \Gamma$ goes
to $+I$ or $-I$ in $\SL(2,\C)$. In the first case the representation
is simply obtained from a homomorphism from $\Gamma/\Z=\pi_1(X)$ into
$\SL(2,\C)$.  The $+I$ case corresponds to Higgs bundles with trivial
determinant as we have seen.  The $-I$ case corresponds to Higgs bundles 
with odd degree determinant. Namely, let
\begin{equation}\label{rgamma}
\mathcal{R}^\pm(\Gamma,\SL(2,\C)) =
\{\rho\in \Hom^{red}(\Gamma,\SL(2,\C)) / \SL(2,\C) \;\;:
\;\;\rho(J)=\pm I\}.
\end{equation}
Here a reductive representation of $\Gamma$ is defined as at the beginning of 
the section, replacing $\pi_1(X)$ by $\Gamma$.
Note that $\mathcal{R}^+(\Gamma,\SL(2,\C))=\calR(\SL(2,\C))$.
We then  have the following \cite{hitchin}.

\begin{theorem}\label{correspondence}
Let $\delta$ be a line bundle over $X$. Then there are homeomorphisms

(i) $\calH(2,\delta)\cong \calR^+(\Gamma,\SL(2,\C))$ if $\deg \delta$ is even,

(ii) $\calH(2,\delta)\cong \calR^-(\Gamma,\SL(2,\C))$ if $\deg \delta$ is odd.

\end{theorem}

\subsection{Fixed points of $\iota(\cO)^{-}$ and representations of $\Gamma$}

For any reductive subgroup $G\subset \SL(2,\C)$ containing $-I$ we consider 
\begin{equation}\label{rgamma}
\mathcal{R}^\pm(\Gamma,G) =
\{\rho\in \Hom^{red}(\Gamma,G) / G \;\;:
\;\;\rho(J)=\pm I\}.
\end{equation}

In particular we have  $\mathcal{R}^\pm(\Gamma,\SU(2))$ and 
$\mathcal{R}^\pm(\Gamma,\SL(2,\R))$. Note that, since $\SU(2)$ is compact,
every representation of $\Gamma$ in $\SU(2)$ is reductive.
We can define the subvarieties  $\mathcal{R}_k^\pm(\Gamma,\SL(2,\R))$ 
of $\mathcal{R}^\pm(\Gamma,\SL(2,\R))$ given by the representations 
of $\Gamma$ in $\SL(2,\R)$ with Euler class $k$. By this, we mean that
the corresponding flat $\PSL(2,\R)$ bundle has Euler class $k$.  
If the $\PSL(2,\R)$ bundle can be lift to an $\SL(2,\R)$ bundle then
$k=2d$, otherwise $k=2d-1$. The 
Milnor inequality \cite{milnor} says that the Euler
class $k$ of any flat $\PSL(2,\R)$ bundle satisfies
$$
|k|\leq 2g-2,
$$
where $g$ is the genus of $X$. 
 
Hitchin proves the following \cite{hitchin}.

\begin{theorem} 
Consider the involution $\iota(\cO)^-$ of $\calH(2,\delta)$. We have 
the following.

(i) The fixed point subvariety of $\iota(\cO)^-$ of points $(E,\varphi)$
with $\varphi=0$ is homeomorphic to the image of   
$\mathcal{R}^\pm(\Gamma,\SU(2))$ in
$\mathcal{R}^\pm(\Gamma,\SL(2,\C))$, where we have $\calR^+$ if the
degree of $\delta$ is even and $\calR^-$ if the degree of $\delta$ is odd. 

(ii) The fixed point subvariety of $\iota(\cO)^-$ of points $(E,\varphi)$
with $\varphi\neq 0$ is homeomorphic to the image of   
$\mathcal{R}^\pm(\Gamma,\SL(2,\R))$ in
$\mathcal{R}^\pm(\Gamma,\SL(2,\C))$, where we have $\calR^+$ if the
degree of $\xi$ is even and $\calR^-$ if the degree of $\xi$ is odd. 

(iii)
More precisely,  the subvariety of triples $\calH_a\subset \calH(2,\delta)$
defined in Section \ref{triples} is homeomorphic
to the image of $\calR_{2a}^+(\Gamma,\SL(2,\R))$ in 
$\mathcal{R}^+(\Gamma,\SL(2,\C))$ if the degree of
$\delta$ is even or to  to the image of $\calR_{2a-1}^-(\Gamma,\SL(2,\R))$ in 
$\mathcal{R}^-(\Gamma,\SL(2,\C))$ if the degree of
$\delta$ is odd. 

\end{theorem}

\begin{proof}

The conjugations with respect to both real forms,
$\SU(2)$ and $\SL(2,\R)$, of  $\SL(2,\C)$ are inner equivalent and
hence they induce the same antiholomorphic involution
of the moduli space $\calR^\pm(\Gamma,\SL(2,C))$, where we recall that the complex structure of
this variety is the one naturally induced by the complex structure of 
$\SL(2,\C)$.
To be more precise, at the level of Lie algebras, the conjugation
with respect to  the real form $\liesu(2)$ is given by the
$\C$-antilinear involution
\begin{equation}\nonumber
  \begin{aligned}
  \tau:\liesl(2,\C)   & \to \liesl(2,\C) \\
    A  &\mapsto -\overline{A}^t,
  \end{aligned}
\end{equation}
while the conjugation  with respect to  the real form $\liesl(2,\R)$ is
given by the  $\C$-antilinear involution
\begin{equation}\nonumber
   \begin{aligned}
  \sigma:\liesl(2,\C)   & \to \liesl(2,\C) \\
    A  &\mapsto \overline{A}.
  \end{aligned}
\end{equation}
Now,
$$
\sigma(A)=J\tau(A)J^{-1}
$$
for $J\in \liesl(2,\R)$ given by
$$
J =
\begin{pmatrix}
  0 & 1 \\
  -1  & 0
\end{pmatrix}.
$$
This is simply because for every $A\in\liesl(2,\R)$, one has that
\begin{equation}\label{n=2}
JA=-A^tJ.
\end{equation}

Under the  correspondence $\calH(2,\delta)\cong\calR^\pm(\Gamma,\SL(2,\C))$, 
the antiholomorphic involution
of $\calR^\pm(\Gamma,\SL(2,\C))$ defined by $\tau$ and $\sigma$
becomes the holomorphic involution $\iota(\cO)^-$ of $\calH(2,\delta)$
\begin{equation}\label{involution}
(E,\varphi)\mapsto (E,-\varphi),
\end{equation}
where we recall that  the  complex  structure of
$\calH(2,\delta)$ is that induced by the complex structure of $X$.
This follows basically from the fact that the 
$\SL(2,\C)$-connection $D$ corresponding to  $(\dbar_E,\varphi)$
under Theorem \ref{correspondence}
is given by (\ref{higgs-connection})
and hence
$$
\tau(D)=\ast_{h}(\dbar_E)+\dbar_E +(\varphi)^{\ast_h}-\varphi,
$$
from which we deduce that $\tau(D)$ is in correspondence with
$(E,-\varphi)$.
Notice also that $\tau(D)\cong \sigma(D)$.

The proof of (i) follows now from the fact that if $\varphi=0$  in
(\ref{higgs-connection}) the connection $D$ is and $\SU(2)$
connection. Note that this reduces to the  
Theorem of Narasimhan and  Seshadri for $\SU(2)$ \cite{narasimhan-seshadri}.

To proof of (ii) and (iii) one easily checks that the connection $D$  defined 
by  a Higgs bundle in $\calH_a(\delta)$ is $\sigma$-invariant and hence
defines an $\SL(2,\R)$-connection.  Now, the Euler class $k$ of the
$\PSL(2,\R)$ bundle is $k=2d$ if $E=L\oplus L^{-1}$, or  
$k=2d-1$ $E=L\oplus L^{-1}\delta$, where  $d=\deg L$. 
\end{proof}

\section{Fixed points of $\iota(\alpha)^\pm$ with
$\alpha\neq \cO$ and representations of $\Gamma$}

Consider the normalizer $N\SO(2)$ of $\SO(2)$ in $\SU(2)$. This is
generated by $\SO(2)$ and 
$J=\begin{pmatrix}
  0 & i \\
  i  & 0
\end{pmatrix}$.
The group generated by $J$ is isomorphic to $\Z/4$
and fits in the exact sequence 
\begin{equation}\label{z4}
0\lra \Z/2\lra \Z/4\lra  \Z/2\lra 1,
\end{equation}
where the subgroup $\Z/2\subset \Z/4$ is $\{\pm I\}$.
We thus have  an exact sequence

\begin{equation}\label{normalizer}
1\lra \SO(2)\lra N\SO(2)\lra  \Z/2\lra 1.
\end{equation}

The normalizer $N\SO(2,\C)$ of $\SO(2,\C)$ in $\SL(2,\C)$
fits also in  an extension
\begin{equation}\label{c-normalizer}
1\lra \SO(2,\C)\lra N\SO(2,\C)\lra  \Z/2\lra 1,
\end{equation}
which is, of course, the complexification of (\ref{normalizer}).

Similarly, we also have that   
$N\SL(2,\R)$, the normalizer of $\SL(2,\R)$ in $\SL(2,\C)$, 
is given by

\begin{equation}\label{normalizer=sl2r}
1\lra \SL(2,\R)\lra N\SL(2,\R)\lra  \Z/2\lra 1.
\end{equation}

Note that $N\SO(2)$ is a maximal compact subgroup of $N\SL(2,\R)$.

Given a representation  $\rho:\Gamma \lra N\SO(2)$ there is a
topological invariant $\alpha\in H^1(X,\Z/2)$, which is given  
by  the map 
$$
 H^1(X,N\SO(2))\lra H^1(X,\Z/2)
$$ 
induced by (\ref{normalizer}).
Let 
$$
\calR_{\alpha}^\pm(\Gamma,N\SO(2)):=\{\rho\in \calR^\pm(\Gamma,N\SO(2))\;\;:\;\;
\mbox{with invariant} \;\;\alpha\in H^1(X,\Z/2)\}.
$$
Similarly, we have this $\alpha$-invariant for representations of $\Gamma$ in
$N\SO(2,\C)$ and in $N\SL(2,\R)$, and we can define  
$\calR_\alpha(\Gamma,N\SO(2,\C))$ and $\calR_{\alpha}^\pm(\Gamma,N\SL(2,\R))$.

\begin{theorem}
Let $\alpha\in J_2(X)=H^1(X,\Z/2)$. Then we have the following.

(i) The subvariety $F_\alpha$ of fixed points of the 
involution $\iota(\alpha)$ in  $M(\delta)$ defined by $E\mapsto E\otimes \alpha$
is homeomorphic to the image of  
$\calR^\pm_\alpha(\Gamma,N\SO(2))$ in $\calR^\pm(\Gamma,\SU(2))$, where we have $\calR^+$ if the
degree of $\delta$ is even and $\calR^-$ if the degree of $\delta$ is odd. 

(ii) The subvariety $\cF_\alpha^+$  of fixed
points of the involution $\iota(\alpha)^+$ of $\calH(\delta)$ is homeomorphic 
to the image of $\calR^\pm_\alpha(\Gamma,N\SO(2,\C))$ in
$\calR^\pm(\Gamma,\SL(2,\C))$, where we have $\calR^+$ if the
degree of $\delta$ is even and $\calR^-$ if the degree of $\delta$ is odd. 

(iii) The subvariety $\cF_\alpha^-$ of fixed
points of the involution $\iota(\alpha)^-$ of $\calH(\xi)$ is homeomorphic 
to the image
of $\calR^\pm_\alpha(\Gamma,N\SL(2,\R))$ in
$\calR^\pm(\Gamma,\SL(2,\C))$, where we have $\calR^+$ if the
degree of $\delta$ is even and $\calR^-$ if the degree of $\delta$ is odd. 

\end{theorem}

\begin{proof}
The element $\alpha\in J_2(X)=H^1(X,\Z/2)$ defines a $\Z/2$ 
\'etale covering  $\pi: X_\alpha\lra X$. The strategy of the 
proof is to lift to $X_\alpha$ and apply a $\Z/2$-invariant
version of the correspondence between Higgs bundles on $X_\alpha$ and
representations of $\Gamma_\alpha$ --- the universal central
extension of $\pi_1(X_\alpha)$. We have a sequence

\begin{equation}\label{gamma-alpha}
1\lra \Gamma_\alpha\lra \Gamma \lra  \Z/2\lra 1.
\end{equation}
since $\Gamma_\alpha$ is the kernel of the homomorphism $\alpha:\Gamma\to \Z/2$ 
defined by $\alpha$.
 
For convenience, let $G$ be any of the  subgroups $\SO(2)\subset \SU(2)$, 
$\SO(2,\C)\subset \SL(2,\C)$
and $\SL(2,\R)\subset \SL(2,\C)$, and let $NG$ be its normalizer in the corresponding  group. We then have an extension

\begin{equation}\label{ng}
1\lra G\lra NG\lra  \Z/2\lra 1.
\end{equation}

Let $\Hom_\alpha(\Gamma,NG)$ be the subset of 
$\Hom(\Gamma,NG)$ consisting of representations of $\rho: \Gamma\to NG$ such 
that  the following diagram is commutative

\begin{equation}\label{commu}
\begin{matrix}
1 & \longrightarrow & \Gamma_\alpha&\longrightarrow &\Gamma &
\stackrel{\alpha}{\longrightarrow} &
\Z/2 &\longrightarrow & 1\\
&& \Big\downarrow && ~\Big\downarrow\rho && \Vert\\
1 & \longrightarrow & G &\longrightarrow & NG &
\longrightarrow &
\Z/2 &\longrightarrow & 1
\end{matrix}
\end{equation}

The group $NG$ is  a disconnected group with
 $\Z/2$ as the group of connected components and $G$ as the
connected component containing the identity. If $G$ is abelian 
($G=\SO(2),\SO(2,\C)$), $\Z/2$ acts on 
$G$ and, since $\Z/2$ acts  on $X_\alpha$ (as the Galois
group) and hence on  $\Gamma_\alpha$, there is thus  an action of $\Z/2$ on
$\Hom(\Gamma_\alpha,G)$. 

A straightforward computation shows that 

\begin{equation}\label{reps-inv-reps}
\Hom_\alpha(\Gamma, NG)\cong \Hom(\Gamma_\alpha,G)^{\Z/2}.
\end{equation}

If $G$ is not abelian (which is the case for 
$G=\SL(2,\R)$), the extension (\ref{ng}) still defines a homomorphism
$\Z/2 \to \Out(G)=\Aut(G)/\Int(G)$. We can then take a splitting of the sequence 

\begin{equation}\label{aut-g}
1\lra \Int(G)\lra \Aut(G) \lra  \Out(G)\lra 1, 
\end{equation}
which always exists \cite{de-siebenthal}. This defines an action  
on $\Hom(\Gamma_\alpha,G)$. However only the action on  
$\Hom(\Gamma_\alpha,G)/G$ is independent of the splitting. In particular, as consequence of (\ref{reps-inv-reps}), we have homeomorphisms  
$$
\calR_\alpha^\pm(\Gamma,NG)\cong \calR^\pm(\Gamma_\alpha,G)^{\Z/2}.
$$

The result follows now from the usual correspondences  between
representations of $\Gamma_\alpha$ and vector bundles or Higgs bundles on 
$X_\alpha$,
combined with the fact that the fixed point subvarieties $F_\alpha$,
$\cF^\pm_\alpha$ described in Sections \ref{prym},  \ref{fix-odd} and 
\ref{fix-even} 
are push-forwards to $X$ of objects  on $X_\alpha$  that 
satisfy the
$\Z/2$-invariance condition (see \cite{garcia-prada-ramanan} for more details).

\end{proof}

\providecommand{\bysame}{\leavevmode\hbox to3em{\hrulefill}\thinspace}

\end{document}